\documentclass[12pt]{article}   
\usepackage{amsthm, amsfonts}
\usepackage{url, cells}
\usepackage{amscd,amsmath,amssymb}

\newtheorem{definition}{Definition}
\newtheorem{theorem}{Theorem}
\newtheorem{lemma}{Lemma}

\newtheorem{corollary}{Corollary}

\def\de{\delta}

\def\eps{\varepsilon}
\def\Si{\Sigma}
\def\al{\alpha}

\def\la{\lambda}

\def\ti{\widetilde}
\def\kappa{\varkappa}
\def\si{\sigma}

\def\C{{\mathbb C}}

\def\sn{{\mathfrak S}_n}

\def\beq{\begin{equation}}
\def\eeq{\end{equation}}
\def\bea{\begin{eqnarray*}}
\def\eea{\end{eqnarray*}}

\def\inv{\operatorname{inv}}
\def\maj{\operatorname{maj}}
\def\comaj{\operatorname{comaj}}
\def\fix{\operatorname{fix}}
\def\des{\operatorname{des}}
\def\Des{\operatorname{Des}}
\def\exc{\operatorname{exc}}

\def\umaj{u_{\maj}}

\def\udes{u_{\des}}
\def\uinv{u_{\inv}}
\def\Mmaj{M_{\maj}}

\def\Mdes{M_{\des}}
\def\Minv{M_{\inv}}

\def\tudes{u_{\widetilde\des}}
\def\tumaj{u_{\widetilde\maj}}
\def\tuinv{u_{\widetilde\inv}}

\def\hmaj{h_{\maj}}
\def\hdes{h_{\des}}
\def\hinv{h_{\inv}}

\def\Reg{\operatorname{Reg}}

\def\Comp{\operatorname{Comp}}
\def\Tr{\operatorname{Tr}}
\def\Ide{\operatorname{Ide}}

\urldef\vershik\url{vershik@pdmi.ras.ru}
\urldef\natalia\url{natalia@pdmi.ras.ru}

\author{N.~V.~Tsilevich\thanks{%
St.~Petersburg Department of Steklov Institute of Mathematics and
St.~Petersburg State University.
E-mail: \natalia. Supported by the RFBR grant 14-01-00373.}
\and A.~M.~Vershik\thanks{%
St.~Petersburg Department of Steklov Institute of Mathematics, 
St.~Petersburg State University, and Institute for Information Transmission Problems.
E-mail: \vershik. Supported by the RSF grant 14-50-00150.}}
\title{On the relation of some combinatorial functions to representation theory}

\date{}

\begin{document}

\maketitle

\begin{abstract}
The paper is devoted to the study of some well-knonw combinatorial functions on the symmetric group $\sn$ --- the major index $\maj$, the descent number $\des$, and the inversion number $\inv$ --- from the representation-theoretic point of view. We show that each of these functions generates in the group algebra the same ideal, and the restriction of the left regular representation to this ideal is isomorphic to the representation of $\sn$ in the space of $n\times n$ skew-symmetric matrices. This allows us to obtain formulas for the functions $\maj,\des,\inv$ 
in terms of matrices of an exceptionally simple form. These formulas are applied to find the spectra of the elements under study in the regular representation, as well as to deduce a series of identities relating these functions to one another and to the number of fixed points
 $\fix$.
 
\end{abstract}

\section{Introduction}

In the representation theory of noncommutative locally compact groups, one of the oldest ideas, which generalizes the idea of Fourier transform, is that the dual object to every element of the group algebra is the operator-valued function defined on the space of equivalence classes of irreducible representations whose value at a given class is the type (up to equivalence) of the unitary operator corresponding to the given element in representations of this class.

Strange as it may seem, this idea, which has been fruitfully employed in the representation theory of Lie groups, has, to the authors' knowledge, little popularity in the representation theory of finite groups and, in particular, of the symmetric group
 $\sn$. In the elaboration of this idea, we can suggest the following definition.

\begin{definition}\label{def:complexity}
 {\rm The {\it dual complexity} of an element  $a$ of the group algebra $\C[G]$ is the dimension of the cyclic subspace (ideal)
 $\Ide(a)=\C[G]a$ generated by all left translations of this element.}
\end{definition}

For example, the dual complexity of every group element (regarded as a $\delta$-function) coincides with  the order of the group, i.e., is the maximum possible. What can be said of elements of the group algebra with small dual complexity? Is not there some dependence, or rather inverse dependence, between the size of the support of an element and its dual complexity, as in the classical theory of Fourier transform? Such questions for the symmetric group directly relate combinatorics to representation theory, and seem not to have been studied. Our paper should be regarded as the first steps in this direction.

The following example was essentially observed (in a quite different connection and in other terms) in 
 \cite{Boj-Szw}. Consider the distance from a permutation $g\in {\frak S}_n$ to the identity of $\sn$ in the word metric with respect to the Coxeter generators, or, in other words, the inversion number $\inv(g)$. The duality of the corresponding element of $\C[\sn]$ is equal to 
 $n(n-1)/2+1$, and the representation in the ideal $\Ide(\inv)$ is isomorphic to the sum of the identity representation and the representation in the space of $n\times n$ skew-symmetric matrices. 

We were interested in more complicated statistics\footnote{According to a tradition which originates in physics, 
functions on the symmetric group are often called statistics of permutations.} on the symmetric group, such as the major index
 $\maj(g)$ and the number of descents $\des(g)$, see, e.\,g.,~\cite{St1}. The study of these statistics goes back to MacMahon
 \cite{MacM}, and by now there is an extensive literature devoted to different problems of enumerative combinatorics involving these functions, which play an important role in the combinatorics of permutations (let us mention, for instance, the papers \cite{Carl, Foata, Sta76, FS, GG, GesReu, SharW}). Note also the recent appearance (in a slightly different version,  for Young tableaux) of the major index 
as a key element of a construction establishing a relation between representations of the infinite symmetric group and the affine algebra 
 $\widehat{\mathfrak{sl}_2}$, see~\cite{serp}. It was an additional motivation for investigating not purely combinatorial, but representation-theoretic properties of the statistics under consideration, which up to now have not been given attention.

The attempt to consider the complexity of these elements in the sense of Definition~{\ref{def:complexity} led to unexpected results: not only the complexity turned out to be also equal to a humble value of $n(n-1)/2+1$, but it appeared that the corresponding ideals
$\Ide(\maj)$ and $\Ide(\des)$ coincide with $\Ide(\inv)$ and thus can also be realized (up to the subspace of constants) in the space of skew-symmetric matrices.

The elaboration of this idea allowed us to obtain for the functions $\maj$, $\des$, $\inv$, originally defined by nontrivial combinatorial conditions, the surprisingly simple formulas~\eqref{maj}--\eqref{inv}  in terms of matrices of a very simple form. The usefulness of these formulas is illustrated by the fact that they allow one to easily find the spectra of the elements under consideration in the regular representation (Theorems~\ref{th:spectra} and~\ref{th:inv}), as well as to obtain a whole series of new simple identities relating these functions to each other and to another important function --- the number of fixed points  
 $\fix$ (Corollaries~\ref{cor:identities},~\ref{cor:fix}).  

In Section~\ref{sec:func} we briefly recall the definitions and basic properties of the statistics under study. The main results are presented in Section~\ref{sec:skew}: we explicitly construct and study in detail the isomorphism between the representations of the group $\sn$ in the ideal $\Ide(\maj)=\Ide(\des)=\Ide(\inv)$ and in the space of skew-symmetric matrices, and obtain formulas for these statistics as simple matrix elements. In Section~\ref{sec:corrs} these formulas are applied to obtain corollaries: to find the spectra of the functions under study in the regular representation and to derive a series of identities involving them. Finally, in Section~\ref{sec:solomon} we briefly describe the relation of our results to investigations of the so-called Solomon descent algebra.

\section{The combinatorial functions $\maj$, $\des$, $\inv$}\label{sec:func}

By $\sn$ we denote the symmetric group of order~$n$, and by
 $\pi_\la$, the irreducible representation of $\sn$ corresponding to a Young diagram $\la$. 


We consider the following functions of permutations $\si\in\sn$ (see, e.\,g., \cite{St1}):
the {\it descent number} 
$$
\des(\si)=\#\Des(\si),\quad \mbox{где}\quad \Des(\si)=\{i=1,\ldots,n-1:\si(i)>\si(i+1)\},
$$ 
the {\it major index}
$$
\maj(\si)=\sum_{i\in\Des(\si)}i,
$$
the {\it inversion number}
$$
\inv(\si)=\#\{i<j:\si(i)>\si(j)\},
$$
and the {\it number of fixed points}
$$
\fix(\si)=\#\{i=1,\ldots,n:\si(i)=i\}.
$$

The statistics $\maj$ and $\inv$ are {\it MacMahonian}, i.e., their generating functions coincide and equal
$$
\sum_{\si\in\sn}q^{\maj(\si)}=\sum_{\si\in\sn}q^{\inv(\si)}=(1+q)(1+q+q^2)\ldots(1+q+\ldots+q^{n-1});
$$
the statistic $\des$ is {\it Eulerian}, i.e., its generating function is given by the Euler polynomials:
$$
\sum_{\si\in\sn}q^{\des(\si)}=A_n(q),\quad\mbox{where}\quad\sum_{n\ge0} A_n(q)\frac{z^n}{n!}=\frac{(1-q)e^z}{e^{qz}-qe^z}.
$$
It is not difficult to deduce that
\beq\label{constants}
\sum_{\si\in\sn}\maj(\si)=\sum_{\si\in\sn}\inv(\si)=n!\cdot\frac{n(n-1)}4,\quad\sum_{\si\in\sn}\des(\si)=n!\cdot\frac{n-1}2.
\eeq

For each of the statistics $\eps=\maj,\des,\inv$, we denote by  $u_\eps$ the corresponding element of the group algebra $\C[\sn]$:
$$
u_\eps=\sum_{g\in\sn}\eps(g)g\in\C[\sn].
$$
We will also need the {\it centered} (i.e., orthogonal to the constants) versions of the statistics under consideration, for which $\sum_{\si\in\sn}\ti\eps(\si)=0$:
\bea
&\widetilde\des(\si)=\des(\si)-\frac{n-1}2,\qquad \widetilde\maj(\si)=\maj(\si)-\frac{n(n-1)}4,&\\
&\widetilde\inv(\si)=\inv(\si)-\frac{n(n-1)}4,&
\eea
and the corresponding elements
$
u_{\widetilde\eps}=\sum_{\si\in\sn}{\widetilde\eps}(\si)\si
$ of the group algebra.

MacMahon studied four fundamental statistics of permutations: $\maj$, $\des$, $\inv$, and also the  {\it excedance number}
$$
\exc(\si)=\#\{i=1,\ldots,n-1:\si(i)>i\}.
$$
This statistic is also Eulerian, i.e., has the same distribution as $\des$. Thus, the four functions $\maj,\inv$ and $\des,\exc$ form two pairs of equally distributed statistics. However, in contrast to $\maj,\des,\inv$, the function $\exc$ generates another ideal, which coincides with the primary component of the natural representation (plus the subspace of constants). Hence its dual complexity is equal to
 $(n-1)^2+1$. It is not difficult to see that the same ideal is generated also by the function $\fix$. 

Note also that sometimes, instead of the major index, one uses the so-called {\it comajor index} $\comaj$, where 
$\comaj(\si)=\sum_{i\in\Des(\si)}(n-i)$. Since, obviously, $\comaj(\si)=n\des(\si)-\maj(\si)$, all the results obtained below for 
$\des$ and $\maj$ can easily be extended to $\comaj$.

\section{Realization of the combinatorial functions $\maj$, $\des$, $\inv$ in the space of skew-symmetric matrices}\label{sec:skew}

\def\M{{\cal M}}
\def\H{{\cal H}}
\def\Pnat{P_1} 
\def\Phook{P_2} 
\def\Cnat{C_1} 
\def\Chook{C_2} 


Let $\M$ be the space of $n\times n$ skew-symmetric matrices. There is a natural action of the symmetric group  $\sn$ in $\M$ by simultaneous permutations of rows and columns. It is well known that this representation $\varrho$ of  $\sn$ decomposes into the sum of two irreducible representations $\pi_{(n-1,1)}$ (the natural representation) and $\pi_{(n-1,1^2)}$, which will be denoted by $\pi_1$ and $\pi_2$, respectively. Also, denote by $\Pnat$ and $\Phook$ the orthogonal projections in $\M$ to the spaces of~$\pi_1$ and~$\pi_2$. 
We equip $\M$ with the inner product
\beq\label{prodinm}
\langle A,B\rangle=\frac12\Tr(AB)=\sum_{i<j}a_{ij}b_{ij},\qquad A=(a_{ij}),\; B=(b_{ij}).
\eeq
The representation $\varrho$ is, obviously, unitary. 

For $1\le i<j\le n$, denote by $E_{ij}$ the skew-symmetric matrix with $1$ in the cell
$(i,j)$ and $-1$ in the cell $(j,i)$, all other entries being zero. Obviously, the matrices $\{E_{ij}\}$ form an orthonormal basis in $\M$ with respect to the inner product~\eqref{prodinm}. By abuse of language, we will call them
the {\it matrix units}.

Consider (also for $1\le i<j\le n$) the following elements of the group algebra $\C[\sn]$:
$$
e_{ij}=\frac1{\sqrt{n!}}\sum\eps_{ij}(\si)\si,\qquad\mbox{where }\eps_{ij}(\si)=\begin{cases}1&\mbox{if }\si(i)<\si(j),\\-1&\mbox{if }\si(i)>\si(j).\end{cases}
$$
Denote by $\H$ the subspace in $\C[\sn]$ spanned by these elements. Let $\Reg_l$ be the left regular representation of the group $\sn$ in $\C[\sn]$.

\begin{lemma}\label{l:isom}
The subspace $\H$  is invariant under $\Reg_l$, and the corresponding subrepresentation is isomorphic to the representation $\varrho$ of~$\sn$ in the space of skew-symmetric matrices $\M$.
\end{lemma}

\begin{proof}\label{l:Ysymm}
Define an operator $\widehat T:\H\to\M$ by the formula $\widehat Te_{ij}=E_{ij}$. It is easy to check that it determines a (nonunitary!) isomorphism of the representations under study. 
\end{proof}

In view of Lemma~\ref{l:isom}, we call the elements $e_{ij}$ the {\it pseudomatrix units} in~$\H$.
It is also convenient to put
$e_{ij}=-e_{ji}$ for  $i>j$ and $e_{ii}=0$. 

The pseudomatrix units are, obviously, not orthogonal with respect to the standard inner product in $\C[\sn]$. Namely, it is not difficult to obtain the following relations: for $i<j$, $k<l$,
$$
\langle e_{ij}, e_{kl}\rangle=\begin{cases}
1,&i=k,\;j=l,\\
\frac13,&i=k, j\ne l\mbox{ or } j=l,i\ne k,\\
-\frac13,&i=l\mbox{ or }j=k,\\
0,&\{k,l\}\cap\{i,j\}=\emptyset.
\end{cases}
$$

Denote by $\Cnat$ and $\Chook$ the normalized central Young symmetrizers (central idempotents in
 $\C[\sn]$) corresponding to the Young diagrams $(n-1,1)$ and $(n-2,1^2)$, respectively. 
 
\begin{lemma}
\beq
\Cnat e_{ij}=\frac1n\sum_{k=1}^n(e_{ik}+e_{kj}).
\eeq 
\end{lemma}

\cellsize=0.8em

\begin{proof}
For $k\ne l$, denote by $c_{kl}$ the Young symmetrizer corresponding to a Young tableau of the form
$\lower1\cellsize\vbox{\footnotesize
\cput(1,1){$k$}
\cput(1,2){}
\cput(1,3){$\ldots$}
\cput(2,1){$l$}
 \cells{
 _ _ _ _ 
|_|_|_|_|
|_|}}$
(the order of the other elements in the first row is irrelevant). It is not difficult to see that
$$
\Cnat=\frac{n-1}{n\cdot n!}\sum_{k\ne l}c_{kl}.
$$
Now one can deduce from the properties of the pseudomatrix units $e_{ij}$ that $c_{kl}e_{ij}=0$ for $\{k,l\}\cap\{i,j\}=\emptyset$ and 
$$
\sum_{k\ne i}c_{ki}e_{ij}=n\cdot(n-2)!\sum_{k=1}^ne_{ik},\quad \sum_{k\ne j}c_{kj}e_{ij}=n\cdot(n-2)!\sum_{k=1}^ne_{kj},
$$
which implies the desired formula.
\end{proof}
 
Since the representation of $\sn$ in the space $\M$ is the sum of two irreducible representations, a unitary isomorphism between the representations in $\H$ and $\M$ can differ from the intertwining operator $\widehat T$ only in that the projections of the elements
$e_{ij}$ to each of the two irreducible components are multiplied by different coefficients, and these coefficients can be found from the condition that the images of the pseudomatrix units should be orthonormal. Namely, consider in $\C[\sn]$ the operator
$$
A
=\frac{\sqrt{3}}{\sqrt{n+1}}(\Cnat+\sqrt{n+1}\Chook)
$$
and put
$$
e'_{ij}=Ae_{ij}.
$$

\begin{theorem}\label{th:isom}
The operator $T:\H\to\M$ defined by the formula
$$
Te_{ij}'=E_{ij},\qquad 1\le i<j\le n,
$$
is a unitary isomorphism of the representations of the symmetric group $\sn$ in the spaces $\H$ and $\M$.
\end{theorem}

\begin{proof}
It is obvious from above that $T$ is an intertwining operator. It remains to show that it is unitary, i.e., that the system of elements
$\{e'_{ij}\}$ in $\C[\sn]$ is orthonormal. Using Lemma~\ref{l:Ysymm}, it is not difficult to calculate that for two different pseudomatrix units
$e_{ij}$ and $e_{kl}$ we have
$$
\langle\Cnat e_{ij},\Cnat e_{kl}\rangle=\begin{cases}
0,&\{k,l\}\cap\{i,j\}=\emptyset,\\
\frac{n+1}{3n},&i=k, j\ne l\mbox{ or } j=l,i\ne k,\\
-\frac{n+1}{3n},&i=l,\;j\ne k\mbox{ or }j=k,\; i\ne l.
\end{cases}
$$
Since $\langle\Chook e_{ij},\Chook e_{kl}\rangle=\langle e_{ij},e_{kl}\rangle-\langle\Cnat e_{ij},\Cnat e_{kl}\rangle$, we obtain that
$$
\langle\Chook e_{ij},\Chook e_{kl}\rangle=\begin{cases}
0,&\{k,l\}\cap\{i,j\}=\emptyset,\\
-\frac{1}{3n},&i=k, j\ne l\mbox{ or } j=l,i\ne k,\\
\frac{1}{3n},&i=l,\;j\ne k\mbox{ or }j=k,\; i\ne l,
\end{cases}
$$
and the fact that the system $\{e'_{ij}\}$ is orthonormal follows by a direct calculation.
\end{proof}

\begin{theorem}\label{th:viaunits}
The centered vectors $\tudes$, $\tumaj$, $\tuinv$ lie in the space $\H$; namely,
$$
\tudes=-\frac{\sqrt{n!}}2\sum_{k=1}^{n-1}e_{k,k+1},\quad\tumaj=-\frac{\sqrt{n!}}2\sum_{k=1}^{n-1}ke_{k,k+1},\quad
\tuinv=-\frac{\sqrt{n!}}2\sum_{1\le i<j\le n}e_{ij}.
$$
\end{theorem}

\begin{proof}
Follows by a straightforward calculation. 
\end{proof}

\begin{corollary}
Each of the vectors $u_\eps$, where $\eps=\des,\maj,\inv$, generates in $\C[\sn]$ the same principal right ideal  ${\cal I}=\H\oplus\{{\rm const}\}$. The restriction of the left regular representation $\Reg_l$ to ${\cal I}$ is the sum of three irreducible representations
$\pi_{(n)}\oplus\pi_{(n-1,1)}\oplus\pi_{(n-2,1^2)}$.
\end{corollary}

Consider in $\M$ the matrices
\bea
&{\scriptsize\hdes=\sum\limits_{k=1}^{n-1}E_{k,k+1}=\begin{pmatrix}
0 & 1&0&\ldots&0\\
   &0          &1         &\ldots&0\\
    &          &0         &\ldots&0\\
    &          &         &\ldots&\\
   &          &         &0&1\\
   &          &         &&0
\end{pmatrix},}\quad
{\scriptsize\hmaj=\sum\limits_{k=1}^{n-1}kE_{k,k+1}=\begin{pmatrix}
0 & 1&0&\ldots&0\\
   &0          &2         &\ldots&0\\
    &          &0         &\ldots&0\\
    &          &         &\ldots&\\
   &          &         &0&n-1\\
   &          &         &&0
\end{pmatrix},} &\\
&{\scriptsize \hinv=\sum\limits_{i<j}E_{ij}=\begin{pmatrix}
0 & 1&1&\ldots&1\\
   &0          &1         &\ldots&1\\
    &          &0         &\ldots&1\\
    &          &         &\ldots&\\
   &          &         &0&1\\
   &          &         &&0
\end{pmatrix}}&
\eea
(when writing matrices, we always indicate only their upper triangular parts, meaning that the lower triangular parts can be recovered from the skew symmetry). Denote $c_n=-\frac{\sqrt{n!}}2$. It follows from Theorems~\ref{th:viaunits}  and~\ref{th:isom} that
\beq\label{uintertw}
\widehat Tu_{\widetilde\eps}=c_nh_\eps,\qquad Tu_{\tilde\eps}=\frac{c_n}{\sqrt{3}}(\sqrt{n+1}\Pnat h_\eps+\Phook  h_\eps).
\eeq

\begin{lemma}\label{l:projunits} 
For $k<m$, the projection of the matrix unit $E_{km}$ to the natural representation is given by the formula
$$
\Pnat E_{km}=(a_{ij})_{i,j=1}^n, \mbox{ where } a_{ij}=\al_i^{(km)}-\al_j^{(km)}\mbox{ and } \al_j^{(km)}=\begin{cases}
\frac1n,&j=k,\\
-\frac1n,&j=m,\\
0&\mbox{otherwise}.
\end{cases}
$$
\end{lemma}

\begin{proof}
It is well known that
\beq\label{decomp}
\begin{aligned}
P_1\M&=\{M=(a_{ij}), \mbox{ where }a_{ij}=\al_i-\al_j \mbox{ and } (\al_j)\in\mathbb R^n,\sum\al_j=0\}, \\
P_2\M&=\{M=(a_{ij}): a_{ji}=-a_{ij},\;\sum_ia_{ij}=0 \mbox{ for every }j\}.
\end{aligned}
\eeq
Using this fact, it is not difficult to find the desired projection.
\end{proof}

\begin{lemma}\label{l:proj}
\bea{\scriptsize
\Pnat\hdes=\frac{1}n\begin{pmatrix}
0 & 1&1&\ldots&1&2\\
   &0          &0         &\ldots&0&1\\
    &          &0         &\ldots&0&1\\
    &          &         &\ldots&&\\
   &          &         &0&0&1\\
   &          &         &&0&1\\
   &          &         &&&0
\end{pmatrix},}\quad
{\scriptsize\Phook\hdes=\frac{1}n\begin{pmatrix}
0 & n-1&-1&-1&\ldots&-1&-2\\
   &0          &n    &0     &\ldots&0&-1\\
    &          &0     &n    &\ldots&0&-1\\
    &          &       &  &\ldots&&\\
    &          &        & &0&n&-1\\
   &          &        & &&0&n-1\\
   &          &        & &&&0
\end{pmatrix};}\\
{\scriptsize\Pnat\hmaj=\begin{pmatrix}
0 & 0&\ldots&0&1\\
   &0                   &\ldots&0&1\\
    &                   &\ldots&0&1\\
    &                   &\ldots&&\\
   &                   &&0&1\\
   &                   &&&0
\end{pmatrix},}\quad
{\scriptsize\Phook\hmaj=\begin{pmatrix}
0 & 1&0&0&\ldots&0&-1\\
   &0          &2    &0     &\ldots&0&-1\\
    &          &0     &3    &\ldots&0&-1\\
    &          &       &  &\ldots&&\\
    &          &        & &0&n-2&-1\\
   &          &        & &&0&n-2\\
   &          &        & &&&0
\end{pmatrix};}
\eea 
$$
\Pnat\hinv=A,\quad 
\Phook\hinv=B,
$$ where $A=(a_{ij})$ and $B=(b_{ij})$ are Toeplitz matrices with the entries
\beq\label{AB}
a_{i,i+k}=\frac{2k}n,\quad b_{i,i+k}=1-\frac{2k}n.
\eeq
\end{lemma}

\begin{proof}
Easily follows from Lemma~\ref{l:projunits} or directly from~\eqref{decomp}.
\end{proof}


Denote by $P$ the orthogonal projection in $\C[\sn]$ to the subspace $\H$, and let $h_0=TP\de_e$. 

\begin{lemma}
$$
h_0=\sqrt{\frac3{(n+1)!}}\left(A+\sqrt{n+1}\cdot B\right),
$$
where the Toeplitz matrices $A$ and $B$ are given by~\eqref{AB}.
\end{lemma}

\begin{proof}
Since $\{e_{km}'\}$ is an orthonormal system in $\H$, we can write $P\de_e=\sum_{km}\langle\de_e,e_{km}'\rangle e_{km}'$. 
Using the above results, we obtain
\begin{multline*}
e_{km}'=\sqrt{\frac{3}{n+1}}\widehat T^{-1}(\Pnat E_{km}+\sqrt{n+1}\Phook E_{km})\\
=\sqrt{\frac{3}{n+1}}\left(\sum_{i<j}(\al_i^{km}-\al_j^{km})E_{ij}+\sqrt{n+1}\sum_{i<j}(\de_{km,ij}-\al_i^{km}+\al_j^{km})E_{ij}\right)\\
=\sqrt{\frac{3}{n+1}}\left(\sqrt{n+1}e_{km}-\sum_{i<j}(\sqrt{n+1}-1)(\al_i^{km}-\al_j^{km})e_{ij}\right).
\end{multline*}
It remains to observe that $\langle\de_e,e_{ij}\rangle=\frac1{\sqrt{n!}}$ for $i<j$, and the desired result follows by straightforward calculations.
\end{proof}

Now we are in a position to prove our key theorem on the representation-theoretic meaning of the combinatorial statistics $\maj,\des,\inv$.

\begin{theorem}\label{th:main}
For every element $f=\sum_{\si\in\sn} c_\si\si\in \H$,
$$
c_\si=\frac1{\sqrt{n!}}\langle\varrho(\si)\hinv,\widehat Tf\rangle.
$$
In particular, for every permutation $\si\in\sn$, the major index $\maj(\si)$, the descent number $\des(\si)$, and the inversion number $\inv(\si)$ 
can be calculated by the following ``matrix'' formulas:
\begin{eqnarray}
\maj(\si)&=&\frac{n(n-1)}4-\frac12\langle\varrho(\si)\hinv,\hmaj\rangle,\label{maj}\\
\des(\si)&=&\frac{n-1}2-\frac12\langle\varrho(\si)\hinv,\hdes\rangle,\\
\inv(\si)&=&\frac{n(n-1)}4-\frac12\langle\varrho(\si)\hinv,\hinv\rangle,\label{inv}
\end{eqnarray}
where $\varrho$ is the representation of the group $\sn$ in the space $\M$ of $n \times n$ skew-symmetric matrices.
\end{theorem}

\begin{proof}
Let $f=\sum_{\si\in\sn} c_\si\si\in \H$ be an arbitrary element from $\H$. For every permutation $\si\in\sn$, we have
$c_\si=\langle \de_\si, f\rangle=\langle\Reg_l(\si)\de_e,f\rangle$. Since $P$ is the orthogonal projection to the invariant subspace
$\H$, we have $\langle\Reg_l(\si)\de_e,f\rangle=\langle\Reg_l(\si)P\de_e,f\rangle$. Acting on both sides by the unitary isomorphism
 $T$, we find
\bea
c_\si=\langle T\Reg_l(\si)P\de_e,Tf\rangle=\langle\varrho(\si)TP\de_e,Tf\rangle=\langle\varrho(\si)h_0,Tf\rangle\\
=\sqrt{\frac{3}{(n+1)!}}\left(\langle\varrho(\si)A,\Pnat Tf\rangle+\sqrt{n+1}\langle \varrho(\si)B, \Phook Tf\rangle\right),
\eea
where we have used the orthogonality of the spaces of different representations.
Note that $P_1Tf=\frac{\sqrt{n+1}}{\sqrt3}\Pnat\widehat Tf$, $P_2Tf=\frac{1}{\sqrt3}\Phook\widehat Tf$, whence 
$$
c_\si=\frac1{\sqrt{n!}}\left(\langle\varrho(\si)A,\Pnat\widehat Tf\rangle+\langle \varrho(\si)B, \Phook\widehat Tf\rangle\right)=
\frac1{\sqrt{n!}}\langle\varrho(\si)\hinv,\widehat Tf\rangle,
$$
again by the orthogonality and since $\hinv=A+B$. The remaining part of the theorem follows from~\eqref{uintertw}.
\end{proof}

\section{Corollaries: spectra and convolutions}\label{sec:corrs}

In this section we show how the ``matrix'' formulas for combinatorial functions found in Theorem~\ref{th:main} allow one to easily obtain results on the spectra of elements (in the regular representation), convolutions, etc.

Denote by $M_\eps$, where $\eps=\maj,\inv,\des$, the operator of right multiplication by
 $u_\eps$ in $\C[\sn]$. Since the operators of left and right multiplication in
 $\C[\sn]$ commute, every eigenspace of  $M_\eps$ is an invariant subspace for the left regular representation $\Reg_l$.

\begin{theorem}\label{th:spectra}
Each of the operators $\Mmaj,\Mdes$ has two nonzero eigenvalues: $s^{\eps}_0>0$ and $s^{\eps}_1<0$, where
\bea
s_0^{\rm maj}=n!\cdot\frac{n(n-1)}4,\qquad s_1^{\rm maj}=-\frac{n!}2;\\
s_0^{\rm des}=n!\cdot\frac{n-1}2,\qquad s_1^{\rm des}=-(n-1)!.
\eea
The correspondent eigenspaces for both operators coincide. The subspace corresponding to  $s^{\eps}_0$ is one-dimensional and coincides with the subspace of constants, i.e., with the subspace of the identity subrepresentation $\pi_{(n)}$ in $\Reg_l$. The subspace corresponding to
 $s^{\eps}_1$ has dimension $\frac{n(n-1)}2$ and coincides with the subspace $\H$ introduced above (in particular, the representation in it is isomorphic to the sum of two irreducible representations
$\pi_{(n-1,1)}+\pi_{(n-2,1^2)}$).  
\end{theorem}

\begin{proof}
Obviously, the constant element ${\mathbf 1}=\sum_{g\in\sn}g$ in $\C[\sn]$ is an eigenvector for $M_\eps$ with the eigenvalue $s^{\eps}_0=\sum_{g\in\sn}\eps(g)$, which can be found from~\eqref{constants}. Hence in what follows we may consider the operators $M_{\ti\eps}$ of multiplication by the centered (orthogonal to the constants) vectors $u_{\ti\eps}$, and we must prove that each of these operators has a single nonzero eigenvalue equal to  $s^{\eps}_1$ and the corresponding eigenspace coincides with $\H$.

The fact that $\H$ is an invariant subspace for $M_{\ti\eps}$ outside of which the operator vanishes immediately follows from Theorem~\ref{th:viaunits}. Each of the irreducible subspaces 
$H_1=\Cnat\H$ and $H_2=\Chook\H$ is also invariant.
Let $f=\sum_{g\in\sn}f(\si)\si\in H_k\subset\H$, where $k=1,2$. By Theorem~\ref{th:main}, we have $f(\si)=\frac1{\sqrt{n!}}\langle\varrho(\si)\hinv,\widehat Tf\rangle$. Then
\bea
(M_{\ti\eps}f)(\si)&=&\sum_{g\in\sn}f(g)\ti\eps(g^{-1}\si)=
-\frac1{2{\sqrt{n!}}}\sum_{g\in\sn}\langle\varrho(g)\hinv,\widehat Tf\rangle\langle\varrho(g^{-1}\si)\hinv,h_{\eps}\rangle\\
&=&-\frac1{2{\sqrt{n!}}}\sum_{g\in\sn}\langle\varrho(g)\hinv,\widehat Tf\rangle\langle\varrho(g)h_{\eps},\varrho(\si)\hinv\rangle\\
&=&-\frac{\sqrt{n!}}{2\dim\pi_k}\langle P_k\hinv,P_kh_{\eps}\rangle\langle\widehat Tf,\varrho(\si)\hinv\rangle
\eea
by the orthogonality relations for matrix elements. But the latter expression is equal to
$-\frac{n!}{2\dim\pi_k}\langle P_k\hinv,P_kh_{\eps}\rangle f(\si)$, which implies that $f$ is an eigenvector of $M_{\ti\eps}$ with the eigenvalue 
\beq\label{eigenvalue}
-\frac{n!}{2\dim\pi_k}\langle P_k\hinv,P_kh_{\eps}\rangle.
\eeq
It remains to observe that $\dim\pi_1=n-1$, $\dim\pi_2=\frac{(n-1)(n-2)}2$ and to calculate the inner products of matrices using Lemma~\ref{l:proj}:
\bea
\langle P_1\hinv,P_1\hdes\rangle=\frac{2(n-1)}n,\quad\langle P_2\hinv,P_2\hdes\rangle=\frac{(n-1)(n-2)}n,\\
\langle P_1\hinv,P_1\hmaj\rangle=n-1,\quad\langle P_2\hinv,P_2\hmaj\rangle=\frac{(n-1)(n-2)}2.
\eea
We see that in the subspaces $H_1$ and $H_2$ the eigenvalues coincide and are equal to the desired value.
\end{proof}

 \begin{corollary}\label{cor:identities}
In the group algebra $\C[\sn]$, the following identities hold for convolutions of the combinatorial functions under consideration:
\bea
\tumaj*\tudes=-(n-1)!\cdot\tumaj,\qquad \tumaj*\tumaj=-\frac{n!}2\cdot\tumaj,\\
\tudes*\tudes=-(n-1)!\cdot\tudes,\qquad \tudes*\tumaj=-\frac{n!}2\cdot\tudes,\\
\tuinv*\tudes=-(n-1)!\cdot\tuinv,\qquad 
\tuinv*\tumaj=-\frac{n!}2\cdot\tuinv,
\eea
or, explicitly,
$$
\sum_{g\in\sn}\widetilde\maj(g)\,\widetilde\des(g^{-1}\si)=-(n-1)!\cdot\widetilde\maj(\si),
$$
and similarly for all the other convolutions. In particular, $\tudes$ and $\tumaj$ are, up to normalization, idempotents in
 $\C[\sn]$.
\end{corollary}

\begin{proof}
By Theorem~\ref{th:spectra}, each of the vectors $\tumaj,\tudes,\tuinv$ lies in an eigenspace of the operators of right multiplication by
 $\tudes,\tumaj$ with known eigenvalues, which implies the desired identities.
\end{proof}

\begin{theorem}\label{th:inv}
The operator $\Minv$ has three nonzero eigenvalues: $s^{\inv}_0>0$ and $s^{\inv}_1,s^{\inv}_2<0$, where
$$
s_0^{\rm inv}=n!\cdot\frac{n(n-1)}4,\qquad s_1^{\rm inv}=-\frac{(n+1)!}6,\qquad s_2^{\rm inv}=-\frac{n!}6.
$$
The subspace corresponding to $s^{\inv}_0$ is one-dimensional and coincides with the subspace of constants. 
The subspace corresponding to $s^{\inv}_1$ has dimension $n-1$ and coincides with  $H_1=\Cnat\H$. 
The subspace corresponding to $s^{\inv}_2$ has dimension $\frac{(n-1)(n-2)}2$ and coincides with  $H_2=\Chook\H$.  
\end{theorem}

\begin{proof}
The proof is entirely similar to that of Theorem~\ref{th:spectra}, one should only calculate
$
\|\Pnat\hinv\|^2=\sum_{k=1}^{n-1}\frac{4k^2}{n^2}(n-k)=\frac{n^2-1}{3} 
$
and $\|\Phook\hinv\|^2=\frac{n(n-1)}{2}-\frac{n^2-1}3=\frac{(n-1)(n-2)}{2}$ and substitute these values into~\eqref{eigenvalue}.
\end{proof}

\begin{corollary}
The function $\inv$ on the group $\sn$ is conditionally nonpositive definite, i.e., for every collection of complex numbers
 $(x_\si)_{\si\in\sn}$,
\beq\label{cpd}
\sum_{g,h\in\sn}\inv(g^{-1}h)x_g\bar x_h\le0  \mbox{ provided that }\sum_{\si\in\sn} x_\si=0.
\eeq
\end{corollary}

Observe that inequality~\eqref{cpd} holds also for  $\maj$ and $\des$, however, these functions, in contrast to $\inv$, are not symmetric on $\sn$ (i.e., do not satisfy the relation $f(\si^{-1})=\overline{f(\si)}$), hence they are not conditionally nonpositive definite according to the classical definition.

The developed techniques allow one to obtain also other identities. Consider, for example, the character $\chi_{\rm nat}$ of the natural representation $\pi_{(n-1,1)}$ of the symmetric group $\sn$. It is well known that $\chi_{\rm nat}(g)=\fix(g)-1$.

\begin{corollary}\label{cor:fix}
\beq\label{fix}
\frac1{n!}\sum_{g\in\sn}\eps(g)(\fix(g)-1)=\begin{cases}-\frac{1}2,&\eps=\maj,\\
-\frac1n,&\eps=\des,\\
-\frac{n+1}6,&\eps=\inv\!.
\end{cases}
\eeq
\end{corollary}

\begin{proof}
Let $\{f_i\}_{i=1}^{n-1}$ be an orthonormal basis in the space $P_1\M$ of the natural representation. Then for
 $\si\in\sn$ we have $\chi_{\rm nat}(\si)=\sum_{i=1}^{n-1}\langle\varrho(\si)f_i,f_i\rangle$. Thus the left-hand side in~\eqref{fix} is equal (again by the orthogonality relations for matrix elements) to
$$
-\frac12\sum_i\frac1{n!}\sum_{\si\in\sn}\langle\varrho(\si)f_i,f_i\rangle\langle\varrho(\si)\hinv,h_\eps\rangle=
-\frac1{2(n-1)}\langle P_1\hinv,P_1h_\eps\rangle,
$$
and~\eqref{fix} follows from the formulas for the inner products obtained in the proofs of Theorems~\ref{th:spectra} and~\ref{th:inv}.
\end{proof}

Note that  identity~\eqref{fix} for the major index can also be obtained from known results on the joint distribution of the statistics
$\maj$ and $\fix$ (see~\cite{GesReu}), but this requires heavy analytic calculations, whereas the matrix formulas give the answer immediately and simultaneously for all three statistics
$\maj,
\des,\inv$.

\section{The Solomon algebra}\label{sec:solomon}

Originally, we proved Theorem~\ref{th:spectra} using the techniques developed in~\cite{GarReu} in connection with the study of the so-called Solomon descent algebra~\cite{Sol}. However, it turned out that the approach presented in Section~\ref{sec:skew} and based on the observations from Theorems~\ref{th:isom} and~\ref{th:viaunits} is much simpler, more efficient and allows one to obtain more results. Nevertheless, we believe that the link to the study in~\cite{GarReu} is important and worth further investigation, so we will briefly describe it in this section.

Denote by $\Comp(n)$ the set of compositions of a positive integer $n$. For $p=(\al_1,\ldots,\al_k)\in\Comp(n)$, put
$$
B_p=\sum_{\si:\Des(\si)\subset\{\al_1,\al_1+\al_2,\ldots,\al_1+\ldots+\al_{k-1}\}}\si\in\C[\sn]. 
$$
In particular,  $B_{(1^n)}=\sum_{\si\in\sn}\si$, and for
$p_k=(1,\ldots,1,2,1,\ldots,1)$ (where~$2$ is in the $k$th position)
$B_{p_k}=\sum_{\si: k\notin\Des(\si)}\si$.

The elements $\{B_p\}_{p\in\Comp(n)}$ generate a subalgebra $\Si_n$ of the group algebra $\C[\sn]$ called the {\it Solomon descent algebra}. In the paper~\cite{GarReu}, devoted to the study of the structure and representations of this algebra, another important basis
$\{I_p\}_{p\in\Comp(n)}$ of  $\Si_n$ was introduced, and the transition matrix between the two bases was obtained. In particular, 
$B_{(1^n)}=I_{(1^n)}$ and $B_{p_k}=I_{p_k}+\frac12I_{(1^n)}$.

Let $a=\sum_qa_qI_q\in\Si_n$, and let $M_a$ be the operator of right multiplication by
$a$ in $\C[\sn]$. In \cite[Theorem~4.4]{GarReu}, it is (implicitly) proved that the eigenvalues  
$s_\la$ of~$M_a$ are indexed by the partitions $\la$ of $n$ and
$s_\la=b_\la\sum_{p:\la(p)=\la}a_p$,
the multiplicity of the eigenvalue $s_\la$ being equal to $\frac{n!}{z_\la}$. 

Then our results on the spectra of the elements
$\udes,\umaj$ follow from the following simple observation.

\begin{lemma}
\bea
\udes=(n-1)B_{(1^n)}-\sum_{k=1}^{n-1}B_{p_k}=\frac{(n-1)}2I_{(1^n)}-\sum_{k=1}^{n-1}I_{p_k},\qquad \\
\umaj=\frac{n(n-1)}2B_{(1^n)}-\sum_{k=1}^{n-1}kB_{p_k}=\frac{n(n-1)}4I_{(1^n)}-\sum_{k=1}^{n-1}kI_{p_k}.
\eea
\end{lemma}

The description of the eigenspaces can also be deduced from the results of  \cite{GarReu} (see also \cite[Theorem~2.2 and Corollary~2.3]{CHS}). However, our approach makes it possible to obtain Theorem~\ref{th:spectra} much easier. Note also that the vector
 $\uinv$ does not lie in the Solomon algebra, hence Theorem~\ref{th:inv} cannot be obtained by the method described in this section.

\end{document}